\documentclass[12pt]{amsart}

\usepackage{amssymb}
\usepackage{enumerate}
\usepackage{bm}

\makeatletter
\@namedef{subjclassname@2010}{%
  \textup{2010} Mathematics Subject Classification}
\makeatother

\newtheorem{thm}{Theorem}[section]

\newtheorem{lem}[thm]{Lemma}
\newtheorem{prop}[thm]{Proposition}

\theoremstyle{definition}
\newtheorem{defin}[thm]{Definition}
\newtheorem{rem}[thm]{Remark}
\newtheorem{exa}[thm]{Example}

\numberwithin{equation}{section}

\newcommand{\bea}{\begin{eqnarray}}
\newcommand{\eea}{\end{eqnarray}}

\newcommand{\clh}{\mathcal{H}}

\newcommand{\cll}{\mathcal{L}}
\newcommand{\clm}{\mathcal{M}}

\newcommand{\cls}{\mathcal{S}}

\newcommand{\clw}{\mathcal{W}}

\newcommand{\raro}{\rightarrow}

\def\textmatrix#1&#2\\#3&#4\\{\bigl({#1 \atop #3}\ {#2 \atop #4}\bigr)}
\def\dispmatrix#1&#2\\#3&#4\\{\left({#1 \atop #3}\ {#2 \atop #4}\right)}
\newcommand{\be}{\begin{equation}}
\newcommand{\ee}{\end{equation}}
\newcommand{\ben}{\begin{eqnarray*}}
\newcommand{\een}{\end{eqnarray*}}

\newcommand{\NI}{\noindent}

\newcommand{\bi}{\begin{itemize}}
\newcommand{\ei}{\end{itemize}}

\newcommand{\D}{\ensuremath{\mathbb{D}}}

\begin{document}

\baselineskip=17pt


\title[Doubly commuting sub-Hardy-modules]{Doubly commuting submodules of the Hardy module over polydiscs}

\author[J. Sarkar]{Jaydeb Sarkar}
\address{J. Sarkar, Indian Statistical Institute\\
Statistics and Mathematics Unit\\
8th Mile, Mysore Road\\
Bangalore, 560059, India}
\email{jay@isibang.ac.in, jaydeb@gmail.com}
\urladdr{http://www.isibang.ac.in/~jay/}

\author[A. Sasane]{Amol Sasane}
\address{A. Sasane, Mathematics Department\\
London School of Economics\\
Houghton Street\\
London WC2A 2AE, U.K.}
\email{sasane@lse.ac.uk}

\author[B. D. Wick]{Brett D. Wick$^\ddagger$}
\address{Brett D. Wick, School of Mathematics\\
Georgia Institute of Technology\\
686 Cherry Street\\
Atlanta, GA USA 30332-0160, U.S.A.}
\email{wick@math.gatech.edu}
\urladdr{www.math.gatech.edu/~wick}
\thanks{$\ddagger$ Research supported in part by National Science Foundation DMS grants \# 1001098 and \# 955432.}

\date{}

\begin{abstract}
In this note we establish a vector-valued version of
Beurling's Theorem (the Lax-Halmos Theorem) for the polydisc.
As an application of the main result, we provide necessary
and sufficient conditions for the ``weak'' completion problem
in $H^\infty(\mathbb{D}^n)$.
\end{abstract}



\maketitle

\section{Introduction and Statement of Main Results}

In \cite{B}, Beurling described all the invariant subspaces for the
operator $M_z$ of ``multiplication by $z$'' on the Hilbert space $H^2(\D)$
of the disc. In \cite{Lax}, Peter Lax extended Beurling's result to
the (finite-dimensional) vector-valued case (while also considering
the Hardy space of the half plane). Lax's vectorial case proof was
further extended to infinite-dimensional vector spaces by Halmos, see
\cite{NF}.  The characterization of $M_z$-invariant subspaces obtained
is the following famous result.

\begin{thm}[Beurling-Lax-Halmos]
  Let $\mathcal{S}$ be a closed nonzero subspace of $H^2_{E_*}(\D)$. Then
  $\mathcal{S}$ is invariant under multiplication by $z$ if and only if there
  exists a Hilbert space $E$ and an inner function $\Theta\in
  H^\infty_{E\rightarrow E_*}(\D)$ such that $ \mathcal{S}=\Theta H^2_E(\D)$.
\end{thm}

For $n\in \mathbb{N}$ and $E_*$ a Hilbert space, $H^2_{E_*}(\D^n)$ is the set of
all $E_*$-valued holomorphic functions in the polydisc $\D^n$, where
${\mathbb{D}}:=\{z\in {\mathbb{C}}:|z|<1\}$ (with boundary
${\mathbb{T}}$) such that
$$
\|f\|_{H^2_{E_*}(\D^n)}:=
\sup_{0<r<1} \Big( \int_{{\mathbb{T}}^n} \| f(r \mathbf{z})\|_{E_*}^2 d \mathbf{z}\Big)^{1/2}
<+\infty.
$$
On the other hand, if ${\mathcal{L}}(E,E_*)$ denotes the set of all
continuous linear transformations from $E$ to $E_*$, then
$H^\infty_{E\rightarrow E_*}(\D^n)$ denotes the set of all
${\mathcal{L}}(E,E_*)$-valued holomorphic functions with
$\|f\|_{H^\infty_{E\rightarrow E_*}(\D^n)}:=\displaystyle\sup_{\mathbf{z}\in
  {\mathbb{D}}^n} \|f(\mathbf{z})\|_{{\mathcal{L}}(E,E_*)}<\infty.$

  An operator-valued $\Theta\in H^\infty_{E\rightarrow E_*}(\D^n)$ {\em inner} if the pointwise a.e. boundary values
  are isometries:
  $$
  (\Theta(\zeta))^*\Theta(\zeta)=I_{E} \textrm{ for almost all }\zeta \in \mathbb{T}^n.
  $$

A natural question is then to ask what happens in the case of several
variables, for example when one considers the Hardy space
$H^2_{E_*}({\mathbb{D}}^n)$ of the polydisc ${\mathbb{D}}^n$.
It is known that in general, a Beurling-Lax-Halmos type
characterization of subspaces of the Hardy Hilbert space is
not possible \cite{Rudin}.   It is however, easy to see that
the Hardy space on the polydisc $H^2_{E_*}(\mathbb{D}^n)$,
when $n >1$, satisfies the \textit{doubly commuting} property,
that is, for all $1 \leq i < j \leq n$
\[
M_{z_i}^* M_{z_j} = M_{z_j} M_{z_i}^*.
\]
We impose this additional assumption to the submodules of
$H^2_{E_*}(\mathbb{D}^n)$ and call that class of submodules
as doubly commuting submodules. More precisely:

\begin{defin}
  A commuting family of bounded linear operators $\{T_1, \ldots,
  T_n\}$ on some Hilbert space $\clh$ is said to be {\em doubly
    commuting} if \[T_i T^*_j = T^*_j T_i,\] for all $1 \leq i, j \leq
  n$ and $i \neq j$.

  A closed subspace $\cls$ of $H^2_{E}(\mathbb{D}^n)$ which is
  invariant under $M_{z_1},\cdots, M_{z_n}$ is said to be a
  \textit{doubly commuting submodule} if $\cls$ is a submodule, that is, $M_{z_i} \cls \subseteq \cls$
  for all $i$ and the
  family of module multiplication operators $\{R_{z_1}, \ldots,
  R_{z_n}\}$ where
  \[R_{z_i} := M_{z_i}|_{\cls},\]for all $1 \leq i \leq n$, is doubly
  commuting, that is, \[R_{z_i} R_{z_j}^* = R_{z_j}^* R_{z_i},\]for
  all $i \neq j$ in $\{1, \ldots, n\}$.
\end{defin}

In this note we completely characterize the doubly commuting
submodules of the vector-valued Hardy module
$H^2_{E_*}(\mathbb{D}^n)$ over the polydisc, and this is the content of our main
theorem.  This result is an analogue of the classical
Beurling-Lax-Halmos Theorem on the Hardy space over the unit disc.

\begin{thm}
\label{Beurling}
Let $\cls$ be a closed nonzero subspace of $H^2_{E_*}(\mathbb{D}^n)$.
Then $\cls$ is a doubly commuting submodule if and only if there
exists a Hilbert space $E$ with $E \subseteq E_*$, where the inclusion
is up to unitary equivalence, and an inner function $\Theta \in
H^{\infty}_{E\to E_*}(\mathbb{D}^n)$ such that
\[\cls = M_{\Theta} H^2_{E}(\mathbb{D}^n).\]
\end{thm}

In the special scalar case $E_*=\mathbb{C}$ and when $n=2$ (the
bidisc), this characterization was obtained by Mandrekar in
\cite{Mandrekar}, and the proof given there relies on the Wold
decomposition for two variables \cite{Slocinski}.  Our proof is
based on the more natural language of Hilbert modules and a
generalization of Wold decomposition for doubly commuting isometries
\cite{Sarkar}.

As an application of this theorem, we can establish a version of the
``Weak'' Completion Property for the algebra $H^\infty(\D^n)$.  Suppose that
$E\subset E_c$.  Recall that the {\em Completion Problem} for
$H^\infty(\D^n)$ is the problem of characterizing the functions $f\in
H^\infty_{E\rightarrow E_c}(\D^n)$ such that there exists an
invertible function $F\in H^\infty_{E_c\rightarrow E_c}(\D^n)$ with
$F\vert_{E}=f$.

In the case of $H^\infty(\D)$, the Completion Problem was settled by
Tolokonnikov in \cite{Tolokonnikov}.  In that paper, it was pointed
out that there is a close connection between the Completion Problem
and the characterization of invariant subspaces of $H^2(\D)$.  Using
Theorem \ref{Beurling} we then have the following analogue of the
results in \cite{Tolokonnikov}.

\begin{thm}[Tolokonnikov's Lemma for the Polydisc]
\label{Tolo}
Let $f\in H^\infty_{E\rightarrow E_c}(\D^n)$ with $E\subset E_c$ and
$\dim E, \dim E_c<\infty$.  Then the following statements are equivalent:
\begin{itemize}
\item[(i)] There exists a function $g\in H^\infty_{E_c\rightarrow
    E}(\D^n)$ such that $gf\equiv I$ in $\D^n$ and the operators
  $M_{z_1}, \ldots, M_{z_n}$ doubly commute on the subspace $\ker M_g$.
\item[(ii)] There exists a function $F\in H^\infty_{E_c\rightarrow
    E_c}(\D^n)$ such that $F\vert_{E}=f$, $F\vert_{E_c\ominus E}$ is
  inner, and $F^{-1}\in H^\infty_{E_c\rightarrow E_c}(\D^n)$.
\end{itemize}
\end{thm}

\begin{rem} Theorem~\ref{Tolo} for the polydisc is different from
  Tolokonnikov's lemma in the disc in which one does not demand that
  the completion $F$ has the property that $F|_{E_c \ominus E}$ is
  inner.  But, from the proof of Tolokonnikov's lemma in the case of
  the disc (see \cite{Nikolski2}), one can see that the following
  statements are equivalent for $f\in H^\infty_{E\rightarrow E_c}(\D)$
  with $E\subset E_c$ and $\dim E<\infty$:
\begin{itemize}
\item[(i)] There exists a function $g\in H^\infty_{E_c\rightarrow
    E}(\D)$ such that $gf\equiv I$ in $\D$.
\item[(ii)] There exists a function $F\in H^\infty_{E_c\rightarrow
    E_c}(\D)$ such that $F\vert_{E}=f$, and $F^{-1}\in
  H^\infty_{E_c\rightarrow E_c}(\D)$.
\item[(ii$^\prime$)] There exists a function $F\in
  H^\infty_{E_c\rightarrow E_c}(\D)$ such that $F\vert_{E}=f$,
  $F|_{E_c \ominus E}$ is inner, and $F^{-1}\in
  H^\infty_{E_c\rightarrow E_c}(\D)$.
\end{itemize}
In the polydisc case it is unclear how the conditions
\begin{itemize}
\item[(II)] There exists a function $F\in H^\infty_{E_c\rightarrow
    E_c}(\D^n)$ such that $F\vert_{E}=f$, and $F^{-1}\in
  H^\infty_{E_c\rightarrow E_c}(\D^n)$.
\item[(II$^\prime$)] There exists a function $F\in
  H^\infty_{E_c\rightarrow E_c}(\D^n)$ such that $F\vert_{E}=f$,
  $F|_{E_c \ominus E}$ is inner, and $F^{-1}\in
  H^\infty_{E_c\rightarrow E_c}(\D^n)$.
\end{itemize}
are related. We refer to the Completion Problem in {\em (II)} as the {\em Strong Completion Problem},
while the one in {\em (II$^\prime$)} as the {\em Weak Completion Problem}. Whether the two are
equivalent is an open problem.

We also remark that in the disc case, Tolokonnikov's Lemma was proved by Sergei Treil \cite{Tre}
without any assumptions about the finite dimensionality of $E, E_c$. However, our proof
 of Theorem~\ref{Tolo} relies on Lemma~\ref{Lemma_on_local_rank}, whose validity we do not know
 without the assumption on the finite dimensionality of $E$ and $E_c$.
\end{rem}

\begin{exa}
\label{example1_6}
As a simple illustration of  Theorem~\ref{Tolo}, take $n=3$, $\dim E=1$, $\dim E_c=3$ and
$$
f:=\left[\begin{array}{ccc} e^{z_1}\\ e^{z_2} \\ e^{z_3} \end{array}\right]\in (H^\infty(\D^3))^{3\times 1}.
$$
With
 $
g:=\displaystyle \left[\begin{array}{ccc} e^{-z_1} &  0 & 0 \end{array}\right]\in (H^\infty(\D^2))^{1\times 3},
$
we see that $gf=1$. We have
\begin{eqnarray*}
\ker M_{g}
&=&
\left\{ \left[\begin{array}{ccc} \varphi_1 \\ \varphi_2 \\ \varphi_3  \end{array}\right]
\in (H^2(\D^3))^{3\times 1}: e^{-z_1} \varphi_1=0\right\}
\\
&=&
\left\{ \left[\begin{array}{ccc} \varphi_1 \\ \varphi_2 \\ \varphi_3  \end{array}\right]
\in (H^2(\D^3))^{3\times 1}: \varphi_1=0\right\}
= \Theta (H^2(\D^2))^{2\times 1},
\end{eqnarray*}
where $\Theta$ is the inner function
$$
\Theta:=\left[\begin{array}{ccc} 0 & 0 \\ 1 & 0 \\ 0 & 1   \end{array}\right]\in  (H^\infty(\D^3))^{3\times 2}.
$$
As $\Theta$ is inner, it follows from Theorem~\ref{Beurling} that
$M_{z_1}, M_{z_2} , M_{z_3}$ doubly commute on  the submodule $\Theta (H^2(\D^3))^{2\times 1} =\ker M_{g}$.
Hence $f$ can be completed to an invertible matrix. In fact, with
$$
F:=\left[\begin{array}{cc} f & \Theta\end{array}\right]
= \left[\begin{array}{ccc} e^{z_1} &  0 & 0  \\
         e^{z_2} &  1 & 0 \\
         e^{z_3} & 0 & 1
        \end{array}\right],
$$
one can easily see that $F$ is invertible as an element of $(H^\infty(\D^3))^{3\times 3}$.
\end{exa}

In Section \ref{s1} we give a proof of Theorem~\ref{Beurling}, and
subsequently, in Section~\ref{s2}, we use this theorem to study the
Weak Completion Problem for $H^\infty(\D^n)$, providing a proof of
Theorem~\ref{Tolo}.

\section{Beurling-Lax-Halmos Theorem for the Polydisc}
\label{s1} In this section we present a complete characterization of
``reducing submodules'' and a proof of the Beurling-Lax-Halmos
theorem for doubly commuting submodules of $H^2_{E}(\D^n)$.

Recall that a closed subspace $\cls \subseteq H^2_{E}(\mathbb{D}^n)$
is said to be a {\em reducing submodule} of $H^2_{E}(\mathbb{D}^n)$
if $M_{z_i} \cls, \,M_{z_i}^* \cls \subseteq \cls$ for all $i = 1,
\ldots, n$.

We start by reviewing some definitions and some well-known facts
about the vector-valued Hardy space over polydisc. For more details
about reproducing kernel Hilbert spaces over domains in
$\mathbb{C}^n$, we refer the reader to \cite{DMS}. Let
\[ \mathbb{S} (\bm{z}, \bm{w})=\prod_{j=1}^n(1-\overline{w}_j z_j)^{-1}. \quad
\quad \quad ((\bm{z}, \bm{w}) \in \mathbb{D}^n \times
\mathbb{D}^n)\] be the Cauchy kernel on the polydisc $\D^n$. Then
for some Hilbert space $E$, the kernel function $\mathbb{S}_E$ of
$H^2_E(\mathbb{D}^n)$ is given by
\[\mathbb{S}_E(\bm{z}, \bm{w}) = \mathbb{S}(\bm{z}, \bm{w}) I_E. \quad \quad \quad ((\bm{z}, \bm{w}) \in
\mathbb{D}^n \times \mathbb{D}^n)\]In particular,
$\{\mathbb{S}(\cdot, \bm{w}) \eta : \bm{w} \in \mathbb{D}^n, \eta
\in E\}$ is a \textit{total subset} for $H^2_E(\mathbb{D}^n)$, that
is,\[\overline{\mbox{span}}\{\mathbb{S}(\cdot, \bm{w}) \eta : \bm{w}
\in \mathbb{D}^n, \eta \in E\} = H^2_E(\mathbb{D}^n),\]where
$\mathbb{S}(\cdot, \bm{w}) \in H^2(\mathbb{D}^n)$ and
\[(\mathbb{S}(\cdot, \bm{w}))(z) = \mathbb{S}(\bm{z}, \bm{w}),\] for all $\bm{z}, \bm{w} \in
\mathbb{D}^n$. Moreover, for all $f \in H^2_E(\mathbb{D}^n)$,
$\bm{w} \in \mathbb{D}^n$ and $\eta \in E$ we have
\[\langle f, \mathbb{S}(\cdot, \bm{w}) \eta \rangle_{H^2_E(\mathbb{D}^n)}
= \langle f(\bm{w}), \eta\rangle_E.\]Note also that for the
multiplication operator $M_{z_i}$ on $H^2_E(\mathbb{D}^n)$
\[M_{z_i}^* (\mathbb{S}(\cdot, \bm{w}) \eta) = \bar{w}_i
(\mathbb{S}(\cdot, \bm{w}) \eta),\]where $\bm{w} \in \mathbb{D}^n$,
$\eta \in E$ and $1 \leq i \leq n$.

We also have
\[\mathbb{S}^{-1}(\bm{z},\bm{w}) = \sum_{0 \leq i_1 < \ldots < i_l \leq n}
(-1)^l z_{i_1} \cdots z_{i_l} \bar{w}_{i_1} \cdots
\bar{w}_{i_l},\]for all $\bm{z}, \bm{w} \in \mathbb{D}^n$.

\NI For $H^2_E(\mathbb{D}^n)$ we set
\[\mathbb{S}_E^{-1}(\bm{M_z}, \bm{M_z}):= \mathop{\sum}_{0 \leq i_1 < \ldots < i_l \leq n} (-1)^l M_{z_{i_1}}
\cdots M_{z_{i_l}} M^*_{{z}_{i_1}} \cdots M^*_{{z}_{i_l}}.\]

The following Lemma is well-known in the study of reproducing kernel
Hilbert spaces.

\begin{lem}\label{P_E}
Let $E$ be a Hilbert space. Then \[\mathbb{S}_E^{-1}(\bm{M_z},
\bm{M_z}) = P_E,\]where $P_E$ is the orthogonal projection of
$H^2_E(\mathbb{D}^n)$ onto the space of all constant functions.
\end{lem}

\begin{proof} for all $\bm{z}, \bm{w} \in \mathbb{D}^n$ and
$\eta, \zeta \in E$ we have
\[
\begin{split} \left\langle \mathbb{S}_E^{-1}(\bm{M_z}, \bm{M_z}) \right.& \left.
(\mathbb{S}(\cdot, \bm{z}) \eta), (\mathbb{S}(\cdot, \bm{w}) \zeta)
\right\rangle_{H^2_E(\mathbb{D}^n)}\\ & = \left \langle
\mathop{\sum}_{0 \leq i_1 < \ldots < i_l \leq n} (-1)^l M_{z_{i_1}}
\cdots M_{z_{i_l}} M^*_{{z}_{i_1}} \cdots M^*_{{z}_{i_l}}
(\mathbb{S}(\cdot, \bm{z}) \eta), (\mathbb{S}(\cdot, \bm{w}) \zeta)
\right\rangle_{H^2_E(\mathbb{D}^n)} \\ & = \mathop{\sum}_{0 \leq i_1
< \ldots < i_l \leq n} (-1)^l \left\langle M^*_{{z}_{i_1}} \cdots
M^*_{{z}_{i_l}} (\mathbb{S}(\cdot, \bm{z}) \eta),  M^*_{{z}_{i_1}}
\cdots
M^*_{{z}_{i_l}} (\mathbb{S}(\cdot, \bm{w}) \zeta) \right\rangle_{H^2_E(\mathbb{D}^n)} \\
& = \mathop{\sum}_{0 \leq i_1 < \ldots < i_l \leq n} (-1)^l
\bar{z}_{i_1} \cdots \bar{z}_{i_l} w_{i_1} \cdots w_{i_l} \langle
\mathbb{S}(\cdot, \bm{z}), \mathbb{S}(\cdot, \bm{w})
\rangle_{H^2(\mathbb{D}^n)} \langle \eta, \zeta \rangle_E \\ & =
\mathbb{S}^{-1}(\bm{w}, \bm{z}) \mathbb{S}(\bm{w}, \bm{z}) \langle
\eta, \zeta \rangle_E \\ & = \left \langle \eta, \zeta
\right\rangle_E
\\& = \langle P_{E} \mathbb{S}(\cdot, \bm{z}) \eta,
\mathbb{S}(\cdot, \bm{w}) \zeta \rangle_{H^2_E(\mathbb{D}^n)}
\end{split}\] Since $\{\mathbb{S}(\cdot, \bm{z}) \eta : \bm{z} \in
\mathbb{D}^n, \eta \in E\}$ is a total subset of
$H^2_{E}(\mathbb{D}^n)$, we have that \[\mathbb{S}_E^{-1}(\bm{M_z},
\bm{M_z}) = P_{E}.\]This completes the proof.
\end{proof}

In the following proposition we characterize the reducing submodules
of $H^2_E(\mathbb{D}^n)$.

\begin{prop}
\label{reducing}
  Let $\cls$ be a closed subspace of $H^2_{E}(\mathbb{D}^n)$. Then
  $\cls$ is a reducing submodule of $H^2_{E}(\mathbb{D}^n)$ if and
  only if \[\cls = H^2_{E_*}(\mathbb{D}^n),\]for some closed subspace
  $E_*$ of $E$.
\end{prop}

\begin{proof} Let $\cls$ be a reducing submodule of $H^2_{E}(\mathbb{D}^n)$, that
is, for all $1 \leq i \leq n$ we have \[M_{z_i} P_{\cls} = P_{\cls}
M_{z_i}.\]By Lemma \ref{P_E}
\[
P_{E} P_{\cls} = \mathbb{S}_E^{-1}(\bm{M_z}, \bm{M_z}) P_{\cls} =
P_{\cls} \mathbb{S}_E^{-1}(\bm{M_z}, \bm{M_z}) = P_{\cls} P_{E}.
\]
In particular, that $P_{\cls} P_{E}$ is an orthogonal projection and
\[P_{\cls} P_{E} = P_{E} P_{\cls} = P_{E_*},\]where $E_* := E \cap
\cls$. Hence, for any
\[
f = \sum_{\bm{k} \in \mathbb{N}^n} a_{\bm{k}}
\bm{z}^{\bm{k}} \in \cls,
\]
where $a_{\bm{k}} \in E$ for all $\bm{k} \in \mathbb{N}^n$, we
have
\[
f = P_{\cls} f = P_{\cls} \left(\,\sum_{\bm{k} \in \mathbb{N}^n}
  M_z^{\bm{k}} a_{\bm{k}}\right) = \sum_{\bm{k} \in \mathbb{N}^n}
M_z^{\bm{k}} P_{\cls} a_{\bm{k}}.
\]
But $P_{\cls} a_{\bm{k}} = P_{\cls} P_{E} a_{\bm{k}} \in
E_*$. Consequently, $M_z^{\bm{k}} P_{\cls} a_{\bm{k}} \in
H^2_{E_*}(\mathbb{D}^n)$ for all $\bm{k} \in \mathbb{N}^n$ and hence
$f \in H^2_{E_*}(\mathbb{D}^n)$.  That is, $\cls \subseteq H^2_{
  E_*}(\mathbb{D}^n)$. For the reverse inclusion, it is enough to
observe that $ E_* \subseteq \cls$ and that $\cls$ is a reducing
submodule. The converse part is immediate. Hence the lemma follows.
\end{proof}

Let $\cls$ be a doubly commuting submodule of $H^2_E(\mathbb{D}^n)$.
Then \[ R_{z_i} R_{z_i}^* = M_{z_i} P_{{\cls}} M_{z_i}^* P_{{\cls}}
= M_{z_i} P_{{\cls}} M_{z_i}^*,
\]implies that $R_{z_i} R_{z_i}^*$ is an orthogonal projection of
$\cls$ onto $z_i \cls$ and hence $I_{\cls} - R_{z_i} R_{z_i}^*$ is
an orthogonal projection of $\cls$ onto $\cls \ominus z_i \cls$,
that is,
\[
I_{{\cls}} - R_{z_i} R_{z_i}^* = P_{{\cls} \ominus z_i {\cls}},
\]
for all $i = 1, \ldots, n$. Define \[\clw_i = \mbox{ran} (I_{\cls} -
R_{z_i} R_{z_i}^*) = \cls \ominus z_i \cls,\] for all $i = 1,
\ldots, n$, and
\[\clw = \bigcap_{i=1}^n \clw_i.\]Now let $\cls$ be a doubly commuting submodule of
$H^2_E(\mathbb{D}^n)$. By doubly commutativity of $\cls$ it follows
that (also see \cite{Sarkar}) \[(I_{{\cls}} - R_{z_i}
R_{z_i}^*)(I_{{\cls}} - R_{z_j} R_{z_j}^*) = (I_{{\cls}} - R_{z_j}
R_{z_j}^*)(I_{{\cls}} - R_{z_i} R_{z_i}^*),\]for all $i \neq j$.
Therefore $\{(I_{{\cls}} - R_{z_i} R_{z_i}^*)\}_{i=1}^n$ is a family
of commuting orthogonal projections and hence
\begin{equation}\label{W-R}\clw = \bigcap_{i=1}^n \clw_i =
\bigcap_{i=1}^n (\cls \ominus z_i \cls)
 = \bigcap_{i=1}^n \mbox{ran} (I_\cls - R_{z_i}
R_{z_i}^*)) =  \mbox{ran}(\mathop{\prod}_{i=1}^n (I_\cls - R_{z_i}
R_{z_i}^*)).\end{equation}

Now we present a wandering subspace theorem concerning doubly
commuting submodules of $H^2_E(\mathbb{D}^n)$. The result is a
consequence of a several variables analogue of the classical Wold
decomposition theorem as obtained by Gaspar and Suciu
\cite{GasparSuciu}. We provide a direct proof (also see Corollary
3.2 in \cite{Sarkar}).

\begin{thm}\label{wandering}
Let $\cls$ be a doubly commuting submodule of $H^2_E(\mathbb{D}^n)$.
Then \[\cls = \mathop{\sum}_{\bm{k} \in \mathbb{N}^n} \oplus
z^{\bm{k}} \clw.\]
\end{thm}
\begin{proof}
First, note that if $\clm$ is a submodule of $H^2_E(\mathbb{D}^n)$
then \[\mathop{\bigcap}_{k \in \mathbb{N}} R_{z_i}^{*k} \clm
\subseteq \mathop{\bigcap}_{k \in \mathbb{N}} M_{z_i}^{*k}
H^2_E(\mathbb{D}^n) = \{0\},\]for each $i = 1, \ldots, n$.
Therefore, $R_{z_i}$ is a shift, that is, the unitary part
$\mathop{\bigcap}_{k \in \mathbb{N}} R_{z_i}^{*k} \clm$ in the Wold
decomposition (cf. \cite{NF}, \cite{Sarkar}) of $R_{z_i}$ on $\clm$
is trivial for all $i = 1, \ldots, n$. Moreover, if $\cls$ is doubly
commuting then
\[R_{z_i} (I_{\cls} - R_{z_j} R^*_{z_j}) = (I_{\cls} - R_{z_j}
R^*_{z_j}) R_{z_i},\]for all $i \neq j$. Therefore $\clw_j$ is a
$R_{z_i}$-reducing subspace for all $i \neq j$. Note also that for
all $1 \leq m <n$, \[\begin{split} \mathop{\bigcap}_{i=1}^{m+1}
\clw_i & = \mbox{ran} (\mathop{\prod}_{i=1}^{m+1} (I_{\cls} -
R_{z_i} R_{z_i}^*))
\\ & = \mbox{ran}(\mathop{\prod}_{i=1}^m (I_{\cls} - R_{z_i} R_{z_i}^*) -
R_{z_{m+1}} R_{z_{m+1}}^* \mathop{\prod}_{i=1}^m (I_{\cls} - R_{z_i}
R_{z_i}^*)) \\ & = \mbox{ran}(\mathop{\prod}_{i=1}^m (I_{\cls} -
R_{z_i} R_{z_i}^*) - R_{z_{m+1}} \mathop{\prod}_{i=1}^m (I_{\cls} -
R_{z_i} R_{z_i}^*)R_{z_{m+1}}^*) \\ & = (\clw_1 \cap \ldots \cap
\clw_m) \ominus z_{m+1} (\clw_1 \cap \cdots \cap
\clw_m),\end{split}\]and hence
\[(\clw_1 \cap \ldots \cap \clw_m) \ominus z_{m+1} (\clw_1 \cap
\cdots \cap \clw_m) = \mathop{\bigcap}_{i=1}^{m+1} \clw_i.\]We use
mathematical induction to prove that for all $2 \leq m \leq n$, we
have \[\cls = \mathop{\sum}_{\bm{k} \in \mathbb{N}^m} \oplus
z^{\bm{k}} (\clw_1 \cap \ldots \cap \clw_m).\] First, by Wold
decomposition theorem for the shift $R_{z_1}$ on $\cls$ we have
\[\cls = \mathop{\sum}_{k_1 \in \mathbb{N}} \oplus R_{z_1}^{k_1} \clw_1 = \mathop{\sum}_{k_1 \in \mathbb{N}} \oplus z_1^{k_1}
\clw_1.\]Again by applying Wold decomposition for $R_{z_2}|_{\clw_1}
\in \cll(\clw_1)$ we have \[\clw_1 = \mathop{\sum}_{k_2 \in
\mathbb{N}} \oplus R_{z_2}^{k_2} (\clw_1 \ominus z_2 \clw_1) =
\mathop{\sum}_{k_2 \in \mathbb{N}} \oplus z_2^{k_2} (\clw_1 \cap
\clw_2),\]and hence \[\cls = \mathop{\sum}_{k_1 \in \mathbb{N}}
\oplus z_1^{k_1} \Big(\mathop{\sum}_{k_2 \in \mathbb{N}} \oplus
z_2^{k_2} (\clw_1 \cap \clw_2)\Big) =  \mathop{\sum}_{k_1, k_2 \in
\mathbb{N}} \oplus z_1^{k_1} z_2^{k_2} (\clw_1 \cap
\clw_2).\]Finally, let \[\cls = \mathop{\sum}_{\bm{k} \in
\mathbb{N}^m} \oplus z^{\bm{k}} (\clw_1 \cap \ldots \cap
\clw_m),\]for some $m < n$. Then we again apply the Wold
decomposition on the isometry
$$
R_{z_{m+1}}|_{\clw_1 \cap \ldots \cap
\clw_m} \in \cll(\clw_1 \cap \ldots \cap \clw_m)
$$
to obtain
\[\begin{split} \clw_1 \cap \ldots \cap \clw_m &  =
\mathop{\sum}_{k_{m+1} \in \mathbb{N}} \oplus z_{m+1}^{k_{m+1}}
\Big((\clw_1 \cap \ldots \cap \clw_m) \ominus z_{m+1} \clw_1 \cap
\ldots \cap \clw_m\Big)\\ & = \mathop{\sum}_{k_{m+1} \in \mathbb{N}}
\oplus z_{m+1}^{k_{m+1}} (\clw_1 \cap \ldots \cap \clw_m \cap
\clw_{m+1}),\end{split}\]which yields \[\cls = \mathop{\sum}_{\bm{k}
\in \mathbb{N}^{m+1}} \oplus z^{\bm{k}} (\clw_1 \cap \ldots \cap
\clw_{m+1}).\]This completes the proof.
\end{proof}

We now turn to the proof of Theorem \ref{Beurling}.

\begin{proof}[Proof of Theorem \ref{Beurling}]

By Theorem \ref{wandering} we have \begin{equation}\label{S=}\cls =
\mathop{\sum}_{\bm{k} \in \mathbb{N}^{n}} \oplus z^{\bm{k}}
(\mathop{\bigcap}_{i=1}^n \clw_i).\end{equation}

Now define the Hilbert space $E$ by \[E = \mathop{\bigcap}_{i=1}^n
\clw_i,\] and the linear operator $V : H^2_{E}(\mathbb{D}^n) \raro
H^2_{E_*}(\mathbb{D}^n)$ by
\[
V\left(\,\sum_{\bm{k} \in \mathbb{N}^n} a_{\bm{k}} z^{\bm{k}}\right) =
\sum_{\bm{k} \in \mathbb{N}^n} M_{z}^{\bm{k}} a_{\bm{k}},
\]
where
$$
\sum_{\bm{k} \in \mathbb{N}^n} a_{\bm{k}} z^{\bm{k}} \in
H^2_{E}(\mathbb{D}^n)
$$
and $a_{\bm{k}} \in E$ for all $\bm{k} \in \mathbb{N}^n$. Observe
that \[\begin{split}\|\sum_{\bm{k} \in \mathbb{N}^n} M_{z}^{\bm{k}}
a_{\bm{k}}\|^2_{H^2_{E_*}(\mathbb{D}^n)} & = \|\sum_{\bm{k} \in
\mathbb{N}^n} {z}^{\bm{k}} a_{\bm{k}}\|^2_{H^2_{E_*}(\mathbb{D}^n)}
= \sum_{\bm{k} \in
\mathbb{N}^n} \|{z}^{\bm{k}} a_{\bm{k}}\|^2_{H^2_{E_*}(\mathbb{D}^n)}, \\
\end{split}
\]where the last equality follows from the orthogonal decomposition
of $\cls$ in (\ref{S=}). Therefore, \[\begin{split}\|\sum_{\bm{k}
\in \mathbb{N}^n} M_{z}^{\bm{k}}
a_{\bm{k}}\|^2_{H^2_{E_*}(\mathbb{D}^n)} & = \sum_{\bm{k} \in
\mathbb{N}^n} \|{z}^{\bm{k}}
a_{\bm{k}}\|^2_{H^2_{E_*}(\mathbb{D}^n)} =  \sum_{\bm{k} \in
\mathbb{N}^n} \| a_{\bm{k}}\|^2_{H^2_{E_*}(\mathbb{D}^n)} =
\sum_{\bm{k} \in \mathbb{N}^n} \| a_{\bm{k}}\|^2_{E} \\& =
\|\sum_{\bm{k} \in \mathbb{N}^n} {z}^{\bm{k}}
a_{\bm{k}}\|^2_{H^2_{E}(\mathbb{D}^n)},
\end{split}\]and hence $V$ is an isometry. Moreover, for all
$\bm{k}\in \mathbb{N}^n$ and $\eta \in E$ we have \[V M_{z_i}
(z^{\bm{k}} \eta) = V (z^{\bm{k} + \bm{e}_i} \eta) = M_z^{\bm{k} +
\bm{e}_i} \eta = M_{z_i} (M_z^{\bm{k}} \eta) = M_{z_i} V (z^{\bm{k}}
\eta),\]that is, $V M_{z_i} = M_{z_i} V$ for all $i = 1, \ldots, n$.
Hence $V$ is a module map. Therefore, \[ V = M_{\Theta},
\]
for some bounded holomorphic function $\Theta \in H^{\infty}_{E\to
  E_*}(\mathbb{D}^n)$ (cf. page 655 in \cite{BLTT}). Moreover, since
  $V$ is an isometry, we have \[M_\Theta^* M_\Theta =
  I_{H^2_{E}(\mathbb{D}^n)},\]that is, that $\Theta$ is an inner
  function. Also since $M_{z_i} E \subseteq \cls$ for all $i = 1, \ldots, n$ we
have that \[\mbox{ran} V \subseteq \cls.\]Also by (\ref{S=}) that
$\cls \subseteq \mbox{ran} V$. Hence it follows that
\begin{equation*}\label{S=V}\mbox{ran} V = \mbox{ran} M_\Theta =
\cls,\end{equation*}that is, \[\cls = \Theta H^2_E(\mathbb{D}^n).\]
Finally, for all $i = 1, \ldots, n$, we have \[\cls \ominus z_i \cls
= \Theta H^2_E(\mathbb{D}^n) \ominus z_i \Theta H^2_E(\mathbb{D}^n)
= \{ \Theta f : f \in H^2_{E}(\mathbb{D}^n), M_{z_i}^* \Theta f =
0\},\]and hence by (\ref{W-R})
\[\begin{split} E & = \bigcap_{i=1}^n \clw_i = \bigcap_{i=1}^n (\cls
\ominus z_i \cls) = \{\Theta f : M_{z_i}^* \Theta f = 0, f \in
H^2_{E}(\mathbb{D}^n), \forall i = 1, \ldots, n\} \\ & \subseteq \{g
\in H^2_{E_*}(\mathbb{D}^n) : M_{z_i}^*g = 0, \forall i = 1, \ldots,
n\} = E_*,\end{split}\]that is, \[E \subseteq E_*.\]

To prove the converse part, let $\cls = M_{\Theta}
H^2_{E}(\mathbb{D}^n)$ be a submodule of $H^2_{E_*}(\mathbb{D}^n)$ for
some inner function $\Theta \in H^{\infty}_{E\to
  E_*}(\mathbb{D}^n)$. Then \[P_{\cls} = M_{\Theta} M_{\Theta}^*,\]
and hence for all $i \neq j$,
\[
\begin{split} M_{z_i} P_{\cls} M_{z_j}^* & = M_{z_i} M_{\Theta}
M_{\Theta}^* M_{z_j}^* = M_{\Theta} M_{z_i}
M_{z_j}^* M_{\Theta}^* = M_{\Theta} M_{z_j}^* M_{z_i} M_{\Theta}^*\\
&  = M_{\Theta} M_{z_j}^* M_{\Theta}^* M_{\Theta} M_{z_i}
M_{\Theta}^* =
M_{\Theta} M_{\Theta}^* M_{z_j}^* M_{z_i} M_{\Theta} M_{\Theta}^*\\
& = P_{\cls} M_{z_j}^* M_{z_i} P_{\cls}.
\end{split}
\]
This implies
\[
R_{z_j}^* R_{z_i} = P_{\cls} M_{z_j}^* P_{\cls} M_{z_i}|_{\cls} =
P_{\cls} M_{z_j}^* M_{z_i}|_{\cls} = M_{z_i} P_{\cls} M_{z_j}^* =
R_{z_i} R_{z_j}^*,
\]
that is, $\cls$ is a doubly commuting submodule. This completes the
proof.
\end{proof}

\section{Tolokonnikov's Lemma for the Polydisc}
\label{s2}

We will need the following lemma, which is a polydisc version of a
similar result proved in the case of the disc in Nikolski's book
\cite{Nikolski2}*{p.44-45}. Here we use the notation $M_g$ for the
multiplication operator on $H^2_E$ induced by $g\in
H^\infty_{E\rightarrow E_*}$.

\begin{lem}[Lemma on Local Rank]
\label{Lemma_on_local_rank}
  Let $E, E_c$ be Hilbert spaces, with $\dim E , \dim E_c<\infty$. Let $g\in
  H^\infty_{E_c \to E}(\D^n)$ be such that
\[
\ker M_g =\{h \in H^2_{E_c}(\D^n)\;:\; g(z)h(z)\equiv 0\}=\Theta
H^2_{E_a} (\D^n),
\]
where $E_a$ is a Hilbert space and $\Theta$ is a
$\mathcal{L}(E_a,E_c)$-valued inner function. Then
\[
\dim E_c= \dim E_a + \textrm{\em rank}\;g,
\]
where $\textrm{\em rank}\;g:= \displaystyle \max_{\zeta \in \D^n} \textrm{\em
  rank}\; g(\zeta)$.
\end{lem}
\begin{proof} We have $\ker M_g=\{ h  \in H^2_{E_c}(\D^n): gh\equiv 0\}$.
If $\zeta \in \D^n$, then let
\[
[\ker M_g](\zeta):= \{ h(\zeta)\;:\; h \in \ker M_g\}.
\]
We claim that $[\ker M_g](\zeta)=\Theta (\zeta) E_a$. Indeed, let $v\in [\ker M_g](\zeta)$. Then
$v=h(\zeta)$ for some element  $h\in \ker M_g=\Theta H^2_{E_a}(\D^n)$. So $h=\Theta\varphi$, for some
$\varphi\in H^2_{E_a}(\D^n)$. In particular, $v= h(\zeta)=\Theta(\zeta)\varphi(\zeta)$, where
$\varphi(\zeta)\in E_a$. So
\begin{equation}
\label{inclu_1_a}
 [\ker M_g](\zeta)\subset \Theta(\zeta) E_a.
\end{equation}
On the other hand, if $w\in \Theta(\zeta)E_a$, then $w=\Theta(\zeta) x$, where $x\in E_a$. Consider
the constant function $\mathbf{x} $ mapping $\D\owns \mathbf{z} \stackrel{\mathbf{x}}{\mapsto} x\in E_a$.
Clearly $\mathbf{x}\in H^2_{E_a}(\D^n)$. So $h:=\Theta \mathbf{x} \in \Theta H^2_{E_a}(\D^n)=\ker M_g$. Hence
$w=\Theta(\zeta) x=(\Theta \mathbf{x})(\zeta)=h(\zeta)$, and so $w\in [\ker M_g](\zeta)$. So we also have that
\begin{equation}
\label{inclu_1_b}
\Theta(\zeta) E_a\subset  [\ker M_g](\zeta).
\end{equation}
Our claim that $[\ker M_g](\zeta)=\Theta (\zeta) E_a$ follows from \eqref{inclu_1_a} and \eqref{inclu_1_b}.

Suppose that for a $\zeta \in
\D^n$, $v\in [\ker M_g](\zeta)$. Then
$v=h(\zeta) $ for some $h\in \ker M_g$. Thus $gh\equiv 0$ in $\D^n$, and in particular,
$g(\zeta) v= g(\zeta) h(\zeta)=0$. Thus $v\in \ker g(\zeta)$. So we have that
$[\ker M_g](\zeta) \subset \ker g(\zeta)$. Hence $\dim [\ker M_g](\zeta)\leq \dim \ker g(\zeta)$.
 Consequently
\[
\dim \Theta(\zeta)E_a
=
\dim [\ker M_g](\zeta)
\leq
\dim \ker g(\zeta)
=
\dim E_c-\textrm{rank}\; g(\zeta),
\]
where the last equality follows from the Rank-Nullity Theorem.
 Since $\Theta$ is inner, we have that the boundary values of $\Theta$ satisfy
 $\Theta(\zeta)^*\Theta(\zeta)=I_{E_c}$ for almost all
$\zeta \in {\mathbb{T}}^n$. So there is an open set $U\subset \D^n$ such that
for all  $\zeta\in U$
$$
\dim E_a= \dim \Theta(\zeta)E_a.
$$
But from the definition of the rank of $g$, we know that
there is a $\zeta_*\in \D^n$ such that we have $k:=\textrm{rank}\; g=\textrm{rank}\; g(\zeta_*)$.
So there is a $k\times k$ submatrix of $g(\zeta_*)$ which is invertible. Now look
at the determinant of this $k\times k$ submatrix of $g$. This is a holomorphic function,
and so it cannot be identically zero in the open set $U$. So there must exist a point $\zeta_1 \in U\subset \D^n$
such that $\textrm{rank}\; g=\textrm{rank}\; g(\zeta_1)$ and $\dim E_a= \dim \Theta(\zeta_1)E_a$.
Hence $\dim E_a \leq
\dim E_c -\textrm{rank}\; g$.

For the proof of the opposite inequality, let us consider a principal
minor $g_1 (\zeta_1)$ of the matrix of the operator $g(\zeta_1)$ (with
respect to two arbitrary fixed bases in $E_c$ and $E$ respectively).
Then $\det g_1\in H^\infty$, $\det g_1 \not\equiv 0$. Let $E_c=E_{c,1}
\oplus E_{c,2}$, $E=E_{1}\oplus E_{2}$ ($\dim E_{c,1}= \dim E_{1}
=\textrm{rank}\; g(\zeta_1)$) be the decompositions of the spaces $E_c$
and $E$ corresponding to this minor, and let
\[
g(\zeta)= \left[ \begin{array}{cc} g_1(\zeta)& g_2 (\zeta)\\
\gamma_1(\zeta) & \gamma_2 (\zeta) \end{array} \right] , \quad \zeta \in \D^n,
\]
be the matrix representation of $g(\zeta)$ with respect to this
decomposition. Owing to our assumption on the rank, it follows that there is a matrix function $\zeta \mapsto W(\zeta)$ such that
$$
\left[ \begin{array}{cc}\gamma_1(\zeta) & \gamma_2 (\zeta) \end{array} \right]
=
W(\zeta)\left[ \begin{array}{cc} g_1(\zeta)& g_2 (\zeta) \end{array} \right].
$$
So $\gamma_2 (\zeta)=W(\zeta) g_2 (\zeta) =(\gamma_1(\zeta) (g_1(\zeta))^{-1}) g_2 (\zeta) $. Thus with
$g_1^{\textrm{co}}:=(\det g_1 ) g_1^{-1}$, we have
\[
\gamma_2 \det g_1 =\gamma_1 g_1^{\textrm{co}} g_2,
\]
and using this we get the inclusion $M_\Omega H^2_{E_{c,2}} (\D^n)\subset \ker M_g$,
where $\Omega \in H^\infty_{E_{c,2} \to E_c}(\D^n)$ is given by
\[
\Omega = \left[ \begin{array}{cc} g_1^{\textrm{co}} g_2 \\ -\det g_1
  \end{array} \right].
\]
We have $
\textrm{rank}\; \Omega =\dim E_{c,2} =\dim E_c- \textrm{rank}\; g=
\dim \ker (g(\zeta_1))$. Consequently, we obtain
  $\dim [\ker M_g](\zeta_1) \geq \dim \ker(g(\zeta_1))$.
\end{proof}

We now turn to the extension of Tolokonnikov's Lemma  to the polydisc.

\begin{proof}[Proof of Theorem \ref{Tolo}]

(ii) $\Rightarrow$ (i): If $g:=P_{E} F^{-1}$, then $gf=I$.  It only
remains to show that the operators $M_{z_1}, \ldots, M_{z_n}$ are
doubly commuting on the space $\ker M_g$.  Let $\Theta$, $\Gamma$ be
such that:
\[
F=\left[ \begin{array}{cc}f & \Theta \end{array} \right]
\quad \textrm{ and } \quad
F^{-1}=\left[ \begin{array}{cc}g \\ \Gamma \end{array} \right] .
\]
Since $F F^{-1} =I_{E_c}$, it follows that $f g +\Theta \Gamma =I_{E_c}$.
Thus if $h \in H^{2}_{E_c} (\D^{n})$ is such that $g h=0$, then $\Theta
(\Gamma h) =h$, and so $h \in \Theta H^{2}_{E_c \ominus E)} (\D^{n} )$.
Hence $\ker M_g \subset \textrm{ran}\;M_\Theta$.  Also, since $F^{-1}F=I$,
it follows that $g\Theta =0$, and so $\textrm{ran}\; M_\Theta \subset \ker M_g$.
So $\ker M_g = \textrm{ran}\; M_\Theta= \Theta H^{2}_{E_c
\ominus E} (\D^{2})$. By Theorem~\ref{Beurling}, the operators $M_{z_1}, \ldots, M_{z_n}$ must doubly
commute on the subspace $\ker M_g$.

\medskip

\noindent (i) $\Rightarrow$ (ii):  Let
$$
\mathcal{S}:=\{h\in H^2_{E_c}(\D^n)\;:\; g(z)h(z)\equiv 0\}=\ker\, g.
$$
$\mathcal{S}$ is a closed non-zero invariant subspace of
$H^2_{E_c}(\D^n)$. Also, by assumption, $M_{z_1},\ldots, M_{z_n}$ are doubly
commuting operators on $\mathcal{S}$.  Then by the above Theorem \ref{Beurling},
there exists an auxiliary Hilbert space $E_a$ and an inner function
$\widetilde{\Theta}$ with values in $\mathcal{L}(E_a,E_c)$ with $\dim
E_a\leq\dim E_c$ such that
$$
\mathcal{S}=\widetilde{\Theta} H^2_{E_a}(\D^n).
$$
By the Lemma on Local Rank, $\dim E_a=\dim
E_c-\textnormal{rank}\; g=\dim E_c-\dim E=\dim (E_c\ominus E)$.  Let
$U$ be a (constant) unitary operator from $E_c\ominus E$ to $E_a$ and
define $\Theta:=\widetilde{\Theta} U$. Then $\Theta$ is inner, and we
have that $\ker g=\Theta H^2_{E_c\ominus E}(\D^n)$.  To get $F\in
H^\infty_{E_c\to E_c}(\D^n)$ define the function $F$ for $z\in\D^n$ by
$$
F(z)e:=
\left\{
\begin{array}{cl}
 f(z)e & \textrm{if } e\in E  \\
 \Theta(z)e & \textrm{if } e\in E_c\ominus E.
\end{array}
\right.
$$
We note that $F\in H^\infty(\D^n)$ and $F\vert_{E}=f$. We now show
that $F$ is invertible. With this in mind, we first observe that
$$
(I-fg)H^2_{E_c}(\D^n)\subset\Theta H^2_{E_c\ominus E}(\D^n)=\ker M_g.
$$
This follows since $g(I-fg)h=gh-gh=0$ for all $h\in H^2_{E_c}(\D^n)$.
Thus we have that $\Theta^*(I-fg)\in H^\infty_{E_c\to E_c\ominus E}(\D^n)$. Now,
define $\Omega=g\oplus \Theta^*(I-fg)$. Clearly $\Omega\in
H^\infty_{E_c\to E_c}(\D^n)$.  Next, note that
$$
F\Omega=fg+\Theta\Theta^*(I-fg)=I.
$$
Similarly,
\begin{eqnarray*}
  \Omega F
  & = &
  gf\mathbb{P}_{E}+\Theta^*(I-fg)(f\mathbb{P}_{E}+\Theta\mathbb{P}_{E_c\ominus E})
  \\
  & = &
  \mathbb{P}_{E}+\Theta^*(f\mathbb{P}_{E}-fgf\mathbb{P}_{E}+\Theta\mathbb{P}_{E_c\ominus E})
  \\
  & = &
  \mathbb{P}_{E}+\Theta^*\Theta\mathbb{P}_{E_c\ominus E}=I.
\end{eqnarray*}
Thus we have that $F^{-1}\in H^\infty(\D^n;E_c\to E_c)$.
\end{proof}

\subsection*{Acknowledgements}

The authors thank
the anonymous referee for the careful review,
for the help in improving the presentation
of the paper, and  also for suggesting Example~\ref{example1_6}.


\begin{thebibliography}{HD}






\normalsize
\baselineskip=17pt



\bibitem[BLTT]{BLTT}
J. Ball, W.S. Li, D. Timotin, T. Trent, 
\emph{A commutant lifting theorem on the polydisc},
   Indiana Univ. Math. J.,
   {48} (1999),
   653-–675.


\bibitem[B]{B} 
 A. Beurling, \emph{On two problems concerning linear transformations in Hilbert
   space}, Acta Math.,
    {81} (1948), 17.


\bibitem[DMS]{DMS} 
R. Douglas, G. Misra and J. Sarkar, 
\emph{Contractive Hilbert modules and their dilations},
    Israel J. Math.,
   {187}, (2012),
    141-–165.


\bibitem[GS]{GasparSuciu} 
 D. Ga{\c{s}}par and N. Suciu, 
\emph{Wold decompositions for commutative families of isometries},
    An. Univ. Timi\c soara Ser. \c Stiin\c t. Mat.,
     {27} (1989),  31--38.


\bibitem[L]{Lax} 
 P.D. Lax, 
 \emph{Translation invariant spaces},
    Acta Math.,
    {101} (1959),
   163--178.

\bibitem[M]{Mandrekar}
V. Mandrekar, \emph{The validity of Beurling theorems in polydiscs},
   Proc. Amer. Math. Soc.,
    {103} (1988), 145--148.


\bibitem[N]{Nikolski2}
N.K. Nikolski{\u\i}, 
\emph{Treatise on the shift operator},
   Grundlehren der Mathematischen Wissenschaften, {273},
    Springer-Verlag, Berlin, (1986).

\bibitem[R]{Rudin}
W. Rudin, 
\emph{Function theory in polydiscs}, 
W. A. Benjamin, Inc., New York-Amsterdam, 
(1969).

\bibitem[Sa]{Sarkar}
 J. Sarkar, \emph{Wold decomposition for doubly commuting isometries},
   preprint, arXiv:1304.7454.


\bibitem[S]{Slocinski}
M. S{\l}oci{\'n}ski, 
\emph{On the Wold-type decomposition of a pair of commuting isometries}, 
Ann. Polon. Math.,
   {37} (1980), 
    255--262.

\bibitem[NF]{NF}
B. Sz.-Nagy and C. Foias, 
\emph{Harmonic analysis of operators on Hilbert space}, 
 North-Holland Publishing Co., Amsterdam (1970).

\bibitem[To]{Tolokonnikov}
V. Tolokonnikov,
   \emph{Extension problem to an invertible matrix},
   Proc. Amer. Math. Soc.,
   {117} (1993),
    1023--1030.

\bibitem[T]{Tre}
S. Treil, 
\emph{An operator Corona theorem},
 Indiana Univ. Math. J., {53} (2004), 
 1763--1780.
	
\end{thebibliography}
\end{document}